\documentclass{amsart}

\usepackage{amsmath,amsthm,amscd,amssymb}
\usepackage{mathrsfs}
\usepackage{graphicx}
\usepackage{enumerate}

\usepackage{ascmac}
\usepackage{color}
\usepackage{here}

\theoremstyle{plain}
\newtheorem{thm}{Theorem}[section]
\newtheorem*{thmA}{Theorem A}
\newtheorem*{thmB}{Theorem B}

\newtheorem{cor}[thm]{Corollary}
\newtheorem{lem}[thm]{Lemma}

\theoremstyle{definition}

\theoremstyle{remark}
\newtheorem{remark}[thm]{Remark}

\numberwithin{equation}{section}

\title[Emergence for diffeomorphisms]{Emergence for diffeomorphisms with nonzero Lyapunov exponents}

\date{\today}

\author{Agnieszka Zelerowicz}
\address[Agnieszka Zelerowicz]{Department of Mathematics, University of California, Riverside, CA-92521, USA}
\email{agnieszz@ucr.edu}

\begin{document}

\begin{abstract}
We consider the set of points with high pointwise
emergence for $C^{1+\alpha}$ diffeomorphisms preserving a hyperbolic measure. We find a lower bound on the Hausdorff dimension of this set in terms of unstable Hausdorff dimension of the hyperbolic measure.
If the measure is an SRB, we prove that the set of points with high emergence has full Hausdorff dimension.  
\end{abstract}

\maketitle

\section{Introduction}\label{section:introduction}

The notion of (metric) emergence has been first introduced by Berger in \cite{Berger2017} as a tool to 'evaluate the complexity to
approximate a system by statistics'.
Metric emergence quantifies such phenomenas as the Newhouse phenomenon or KAM phenomenon. The list of research works on metric emergence includes \cite{Berger2020, BBo, BBi, Talebi, CJZ}.
Following similar ideas, Nakano, Kiriki, and Soma \cite{KNS2019}, introduced a concept of pointwise emergence (see (\ref{eq:0403d}) for the definition) to measure complexity of irregular orbits. 
We recall that
a point $x$ is said to be \emph{irregular}, if the sequence
 \begin{equation}\label{eqn:bir-average}
\delta _x^n = \frac{1}{n} \sum _{j=0}^{n-1} \delta _{f^j (x)}, \quad n\geq 1,
\end{equation}
 does not converge (here $\delta_x$ denotes the atomic measure at the point $x$).

By the Birkhoff Ergodic Theorem, the set of irregular points has zero measure with respect to any invariant measure. On the other hand, for many dynamical systems this set is known to be large from different points of view -  to have full Hausdorff dimension and full entropy.  
Pesin and Pitskel' \cite{PP84}   obtained the first result of this type, proving that the set of irregular points for the full shift has full entropy and full Hausdorff dimension. 
This result was later extended to topologically mixing subshifts of finite type in \cite{BS2000}, 
to graph directed Markov systems in \cite{FP2011},
to continuous maps with specification property in \cite{CKS2005} (see also \cite{Thompson2010}), 
and  to continuous maps with almost specification property
in \cite{Thompson2012}.
Other works related to the study of the set of irregular points include \cite{HK1990, Ruelle2001, Takens2008, ABC2011, CZZ2011, CTV2015, Tian2017, BV2017, KS2017, AP2019, BKNSR2020, Yang2020}.

In this paper we consider irregular points with high complexity, that is points for which (\ref{eqn:bir-average}) oscillates between infinitely many ergodic measures.
This corresponds to the pointwise emergence being super-polynomial and we will also refer to it as \textit{high} (for the definition see (\ref{eqn:high}) in the next section).

As pointed out in \cite{Berger2017}, there is a consensus among computer scientists that super-polynomial algorithms are impractical. From that perspective dynamical systems with high metric emergence are not feasible to be studied numerically.
The set of points with high (pointwise) emergence can be then considered as statistically very complex (see also \cite{Berger2020,BBo} for other motivations to study high emergence).

It was proved by the author of this note and Nakano in \cite{NakZ} that the set of points with high emergence for topologically mixing subshifts of finite type has full entropy, full Hausdorff dimension, as well as full pressure for any H\"older continuous function.
The aim of this paper is to find a lower bound for the Hausdorff dimension of the set of points with high emergence for $C^{1+\alpha}$ diffeomorphisms preserving a hyperbolic measure (see Theorem A).
This is done by finding an appropriate approximation by Horseshoes (Theorem B) and then applying the result from \cite{NakZ}.   
The first construction of a horseshoe approximating the support of a hyperbolic measure was presented by Katok in \cite{Katok1980} and since then there have been many works establishing different variants of this classical result.
Examples include \cite{ACW, Chung, Gel, Luz, Mendoza, Mis, Per, Przyt, Sanchez, Yang} but this list is far from complete.
For our purposes we need a horseshoe maximizing the unstable dimension. 
This was done in \cite{Mendoza} for diffeomorphisms on surfaces preserving an ergodic $SRB$ measure and in \cite{Sanchez} this result has been extended to higher dimensions.
Theorem B of this paper is a generalization of  \cite{Mendoza} and  \cite{Sanchez} to general hyperbolic measures and as such, we believe that it may be of independent interest.


\subsection{Structure of the paper}
In Section \ref{sec:setting} we introduce the setting and state the main results - Theorem A and Theorem B. 
We prove Theorem B in Section \ref{sec:Kat}. The proof of Theorem A occupies Sections \ref{sec:emer} and \ref{sec:dim}.

\section*{Acknowledgments}
The author would like to thank Dmitry Dolgopyat for suggesting the problem and for helpful remarks.

\section{The setting and the main results}\label{sec:setting}

Consider a $C^{1+\alpha}$ diffeomorphism $f$ ($\alpha>0$) of a smooth compact Riemannian manifold $M$ without boundary and $\mu$ - an $f$-invariant hyperbolic probability measure.
Let $0 < \chi_1<\chi_2<\ldots<\chi_l$ denote distinct positive Lyapunov exponents of $f$ with multiplicities $n_1, n_2, \ldots, n_l$.
The following statement was proved by Ledrappier and Young in \cite{LY} (see also \cite[Theorem 14.1.18]{pesRed}).

\begin{thm}\label{thm:LY}
There are numbers $d_1, d_2, \ldots, d_l$ such that:
\begin{enumerate}
\item $0\leq d_i\leq n_i$,
\item $\sum_{i=1}^l d_i = d_{\mu}^u$, and
\item $h_{\mu}(f)= \sum_{i=1}^l  d_i ~ \chi_i$.
\end{enumerate}
\end{thm}
Here $ d_{\mu}^u$ denotes the unstable dimension of the measure and can be obtained \cite[Theorem 14.1.6]{pesRed} as the limit

$$  \lim_{r\to 0} \frac{\log \mu^u(x)(B^u(x,r))}{\log r}   \text{ for $\mu$- almost every }x\in M,$$

where $ \mu^u(x)$ is the conditional measure on the local unstable manifold $V^u(x)$ and $B^u(x,r):= B(x,r)\cap V^u(x).$

For the rest of this paper we assume that: 

\begin{enumerate}
\item[\textbf{A1.}]  there exists a number $0 < d \leq 1$ such that    $$  d = \frac{d_1}{n_1} =   \frac{d_2}{n_2}  = \ldots  =  \frac{d_l}{n_l}.$$
\item[\textbf{A2.}] $$\frac{\chi_l}{\chi_1}(1-d)<1.$$ 
\end{enumerate}

In particular our results apply if $\mu$ is an $SRB$ measure, in which case $d=1$. Another example is  
 $\mu$ being \textit{u-conformal}, meaning that there is only one positive Lyapunov exponent $\chi^+$ of multiplicity $n^+\geq 1$.
In the latter case the statement of Theorem \ref{thm:LY} reduces to the formula
$h_{\mu}(f) =  d_{\mu}^u  ~ \chi^{+}$.

The \emph{pointwise emergence} $\mathscr E_{f,x}(\epsilon ) $  at scale $\epsilon >0$ at $x\in M$ (with respect to $f$) is defined as
\begin{multline}\label{eq:0403d}
\mathscr E_{f,x}(\epsilon ) = \min \Big\{ N \in \mathbb N \mid  \text{there exists   $\{ \mu _j\} _{j=1}^N \subset \mathcal P(M)$ such that} \\
\limsup _{n\to \infty}  \min _{1\leq j\leq N}  d_W \left( \sum_{k=0}^{n-1}\delta _{f^k(x)} , \mu _j \right) \leq \epsilon \Big \},
\end{multline}

where:
\begin{itemize}
\item $\mathcal P(M)$ denotes the set of Borel probability measures on $M$,
\item $d_W$ is the first Wasserstein metric:\\
For $j=1,2$, let $\pi_j :M\times M\to M $ be the canonical projection to the $j$-th coordinate. 
 Let $\Pi (\mu, \nu )$ be the set of probability measures $\mathbb{P}$ on $M\times M$ such that $\mathbb{P}\circ \pi_1^{-1} =\mu$ and $\mathbb{P}\circ \pi_2^{-1} =\nu$. 
 The first Wasserstein metric $d_W $ is defined as
 \[
d_W(\mu ,\nu )= \inf _{\mathbb{P} \in \Pi (\mu , \nu )} \int _{M\times M} d(x,y) d\mathbb{P} (x,y) \quad \text{for $ \mu  ,\nu \in \mathcal P(M)$} ,
 \]
\item $\delta_x$ denotes the atomic measure at the point $x$.
\end{itemize}

The pointwise emergence at $x\in M$ is called \emph{super-polynomial} (or \emph{high})
 if 
\begin{equation}\label{eqn:high}
\limsup _{\epsilon \to 0}  \frac{\log \mathscr E_{f,x}(\epsilon ) }{-\log \epsilon } =\infty .
\end{equation}

The following was shown in \cite[Theorem 1.1]{NakZ}.
\begin{thm}\label{thm:mainNakZ}
Let $X $ be a topologically mixing subshift of finite type of $  \{ 1, 2, \ldots , \kappa\} ^{\mathbb N}$ with $\kappa\geq 2$. Let $\sigma:X \to X$ be the left-shift operation on $X$.  
 Let $\mathcal{E}_{\sigma}$ be the set of points $x\in X$ satisfying
 \begin{equation*}
 \lim _{\epsilon \to 0}  \frac{\log \mathscr E_{\sigma,x}(\epsilon ) }{-\log \epsilon } =\infty.
\end{equation*}
Then, 
\[
h_{\mathrm{top}} (\sigma_{|\mathcal{E}_{\sigma}}) =h_{\mathrm{top}} (\sigma_{|X})  \quad \text{and } \quad \mathrm{dim}_{H} (\mathcal{E}_{\sigma}) =\mathrm{dim}_{H} (X) .
\]
In addition, for any H\"older continuous function $\varphi$, we have that $P(\sigma_{|\mathcal{E}_{\sigma}}, \varphi)=P(\sigma_{|X},\varphi)$.
That is, the set of points with high emergence carries full topological pressure.
\end{thm}

The aim of this note is to use the above result to obtain the following.

\begin{thmA}\label{thm:main}
Let $f: M \to M$ be a $C^{1+\alpha}$ diffeomorphism ($\alpha>0$) of a smooth compact Riemannian manifold $M$ without boundary and $\mu$ - an $f$-invariant hyperbolic probability measure
satisfying Conditions \textbf{A1.} and \textbf{A2.}

 Let $\mathcal{E}_{f}$ be the set of points $x\in M$ satisfying
 \begin{equation*}
 \lim _{\epsilon \to 0}  \frac{\log \mathscr E_{f,x}(\epsilon ) }{-\log \epsilon } =\infty.
\end{equation*}
Then,

$$  \dim_H(\mathcal{E}_f) \geq  \dim(E^s) +  \left( 1   - (1-d)\frac{\chi_l}{\chi_1}   \right)\dim E^u ,   $$
where $E^s$ and $E^u$ denote the stable and unstable subbundles respectively.
\end{thmA}

In the two special cases we obtain the following corollaries.

\begin{cor}
Let $f: M \to M$ be a $C^{1+\alpha}$ diffeomorphism of a smooth compact Riemannian manifold $M$ without boundary 
preserving a hyperbolic $SRB$ measure. Then,
$$  \dim_H(\mathcal{E}_f) = \dim(M) .   $$
\end{cor}

\begin{cor}
Let $f: M \to M$ be a $C^{1+\alpha}$ diffeomorphism of a smooth compact Riemannian manifold $M$ without boundary and $\mu$ - an $f$-invariant hyperbolic probability measure
with one positive Lyapunov exponent $\chi^+$ of multiplicity $n^+\geq 1$.
Then, 
$$  \dim_H(\mathcal{E}_f) \geq  \dim(E^s) +  d_{\mu}^u  .   $$
\end{cor}

We will prove Theorem A by constructing an appropriate horseshoe in a neighborhood of the support of $\mu$. 
We prove the following result in Section \ref{sec:Kat} by combining arguments from \cite{ACW} and \cite{Mendoza}.


\begin{thmB}\label{thm:Kat}
Consider a $C^{1+\alpha}$ diffeomorphism $f:M \to M$ preserving a hyperbolic probability measure $\mu$ satisfying Conditions \textbf{A1} and \textbf{A2}. For any $\delta>0$ and $0 <\delta_1  \ll \delta $ there exists a compact $f$-invariant set $\Lambda\subset M$ such that:
\begin{enumerate}
\item there exists a dominated splitting on $\Lambda$:
$$ T_{\Lambda}M = E^s \oplus E^1 \oplus E^2 \oplus \ldots \oplus E^l, \text{ with }\dim(E^i)=n_i \text{ for }i= 1,\ldots, l;   $$
\item there exists $n_0\geq 1$ such that for each $i= 1,\ldots, l$, each $x\in\Lambda$, each unit vector $v\in E^i(x)$ and all $n\geq n_0$,
$$   \exp((\chi_i -\delta_1)n)\leq ||   Df^n_x(v) ||  \leq   \exp((\chi_i+\delta_1)n) ;  $$

$~$
\newline denote  $E^u =  E^1 \oplus E^2 \oplus \ldots \oplus E^l$,

\item the function $   \log |  \det  Df_{| E^{u}}|  $ is H\"older continuous on $\Lambda$;

\item the topological pressure satisfies:  $P(f_{|\Lambda},   - d   \log |  \det  Df_{| E^u}| ) > -\delta $;

\item there exists an ergodic, $f$-invariant measure $\nu$ supported on $\Lambda$ and such that
$$ d^u_{\nu}  \geq  \dim E^u \left(       1 - \frac{\chi_l}{\chi_1}(1-d)  \right)   - \delta  >0 .    $$ 
\end{enumerate}
\end{thmB}

\begin{remark}\label{rmk:1}
A construction of a hyperbolic horseshoe satisfying (1), (2), and with topological entropy close to $h_{\mu}(f)$ was presented in \cite[Theorem 3.3]{ACW}.
Even though it was not stated there, Statement (3) is an immediate consequence of statements (1) and (2), see for example \cite[Theorem 5.3.2]{pesRed} or \cite[Theorem 19.1.6 and Corollary 19.1.13]{Kat}.
For our purposes we need a horseshoe with unstable Hausdorff dimension close to $d^u_{\mu}$. This can be obtained by replacing in \cite{ACW} quantities corresponding to entropy with analogous quantities corresponding to pressure of an appropriate potential.
We will then show that (4) implies (5).
\end{remark}


\section{Proof of Theorem B}\label{sec:Kat}

\subsection{Summary of the proof}
We start the proof by recalling some classical results from the theory of non-uniformly hyperbolic systems. As the first step in proving Theorem B we establish existence of a compact (but not $f$-invariant) set $X$ of nearly full measure and
on which the dynamics of $f$ is uniformly hyperbolic. That is, the set $X$ satisfies Statements $(1)$ and $(2)$ of Theorem B but it lacks invariance.

Next we define and study some thermodynamical quantities on $X$. The results obtained in this step are stated in Lemma 3.3 and will ultimately contribute to proving Statement (4) and (5) in Theorem B.

In Lemma 3.4 we use recurrence and Lemma 3.3 to construct a finite subset $\tilde{X}\subset{X}$ of points that start in the same, arbitrarily small neighborhood $B$ and return to $B$ after some large but fixed time $L\in\mathbb{N}$.
The set $\tilde{X}$ is constructed in such a way, that between times $0$ and $L$ the trajectories of points in $\tilde{X}$ "separate"\footnote{We give precise definition of a \textit{separated set} in Section \ref{sec:thermodynamic}.} 
by some small distance $\rho>0$ and in fact, the orbit segments of length $L$ of all the points in $\tilde{X}$ collectively approximate the topological pressure of
$\phi=-d \log | \det Df_{| E^u}|$. 

Finally we introduce $\Lambda$ - the set of all points that shadow infinite sequences of trajectory segments (of length $L$) of points in $\tilde{X}$.
We prove that this set satisfies the assertion of Theorem B. In particular the measure $\nu$ is obtained as the unique equilibrium measure for $\phi$ on $\Lambda$.

\subsection{Proof of Theorem B}\label{sec:Kat}
We need the following Oseledec-Pesin reduction theorem for the derivative cocycle  \cite[Theorem 5.6.1 and Theorem 3.5.5]{pesRed}. 

\begin{thm}\label{thm:Pes}
Let $f : M \to M$ be a $C^{1+\alpha}$ diffeomorphism preserving an ergodic probability measure $\mu$
with Lyapunov exponents $\lambda_1 > \ldots > \lambda_s$ of multiplicities $k_1, \ldots k_s$.
Then there exists an invariant set $Z\subset M$ with $\mu(Z)=1$ such that for any $\eta>0$ the following holds on $Z$:
\begin{enumerate}
\item There exists a measurable family of invertible linear maps $C(x): T_xM \to \mathbb{R}^{\dim M}$ and $A_i(x)\in GL(k_i, \mathbb{R})$, $i=1, \ldots ,s,$ such that

\begin{enumerate}
\item\label{eqn:Oseledecsplit}
$$Df(x)= C^{-1}(f(x)) diag[ A_1(x), \ldots, A_s(x)  ] C(x),$$

\item\label{eqn:norms}
$$  e^{\lambda_i-\eta} < ||A_i^{-1}||^{-1} \leq ||A_i|| < e^{\lambda_i+\eta} ,  $$ 

\item\label{eqn:temp} 
$C(x)$ is tempered, that is for all $x\in Z$ $$  \lim_{m\to \pm \infty  }  \frac{1}{m}   \log ||   C(f^m(x))  ||   =    \lim_{m\to \pm \infty  }  \frac{1}{m}   \log ||   C^{-1}(f^m(x))  ||  =0;      $$
\end{enumerate}

\item there exist measurable functions $r,K: Z \to (0,1],$ and a collection of embeddings $\psi_x:  B(0, r(x)) \to M$ such that:

\begin{enumerate}
\item   $\psi_x(0) = x $ and $\psi_x = \exp_x \circ C(x),$

\item the maps $f_x = \psi_{f(x)}    \circ f   \circ \psi^{-1}_x$ satisfy   $d_{C^1}(f_x, Df_x(0))<\eta,$

\item there exists a constant $Q>0$ such that  for any $y,y'\in   \psi^{-1}_x(B(0,r(x))),$
$$      \frac{1}{K(x)}d(y,y')   \leq    ||  \psi_x(y) -  \psi_x(y')  ||   \leq    \frac{1}{Q}d(y,y')  ,$$

\item  $$  e^{-\eta}  < \frac{r(f(x))}{r(x)} < e^{\eta}  \text{ and } e^{-\eta}  < \frac{K(f(x))}{K(x)} < e^{\eta}  .$$
\end{enumerate}
\end{enumerate}
\end{thm}

Statement (\ref{eqn:temp}) guarantees that choosing $m_0\in\mathbb{N}$ large enough, we can find a set $X_0\subset Z$ of measure arbitrarily close to one on which $     ||   C\circ f^m||, ||C^{-1}\circ f^m||\in  ( e^{-\eta m}, e^{\eta m})$ for all $m\geq m_0$.
This combined with Lusin's theorem gives that for any $\eta>0$ there exists a compact set $X \subset X_0 \subset Z$ of measure arbitrarily close to one on which the functions $r,K,C, C^{-1}$ are continuous.
Following the terminology used in \cite{ACW}, we will refer to any such set as a \textit{uniformity block of tolerance $\eta$}. It is worth pointing out that those are closely related to Pesin's regular sets.

\subsubsection{Thermodynamic quantities}\label{sec:thermodynamic}

For $n\geq 1$ we define the \textit{dynamical metric} $d_n$ on $M$ as
$$  d_n(x,y):=  \max_{0\leq k\leq n} d(f^k(x),f^k(y)) . $$
For $x\in M$ and $\rho>0$ we denote by $B_n(x,\rho)$ the open ball with respect to the metric $d_n$ centered at $x$ of radius $\rho$.
We say that a set $E\subset M$ is
\begin{itemize} 
\item \textit{$\mu-(n,\rho,\beta)$- spanning} if $ \mu   ( \bigcup_{x\in E}  B_n(x,\rho) )   > 1-\beta.$
\item \textit{$(n,\rho)$- separated} if for any two elements $x,x'\in E$ one has $d_n(x,x')>\rho$.
\newline In addition, for any subset $A\subset M$ and an $(n,\rho)$- separated set $E\subset A$ we say that 
\item $E$ is a \textit{maximal $(n,\rho)$- separated set in  $A$}, if there is no other  $(n,\rho)$- separated set in  $A$ containing $E$.
\end{itemize}

We observe that for any $A\subset M$ one has that if $E$ is a maximal $(n,\rho)$- separated set in  $A$, then $E$ is $\mu-(n,2 \rho,1-\mu(A))$- spanning.
We denote by $C(n,\rho,\beta)$ the minimal cardinality of a $\mu-(n,\rho,\beta)$- spanning set. The following was shown in \cite[Theorem 1.1]{Katok1980}.

\begin{thm}\label{thm:entr}
For any $\beta \in (0,1) ,$
$$  h_{\mu}(f) =      \lim_{\rho\to 0}    \lim_{n\to \infty}  \frac{1}{n} \log   C(n,\rho,\beta).$$
\end{thm}

We now consider the function $$  \phi(x): =    -d \log |  \det Df_{| E^u(x)} |  \text{ and denote } S_n\phi(x) = \sum_{k=1}^{n-1}\phi (f^k(x)).  $$ Observe that $\phi$ is only measurable, however it is continuous on every uniformity block.
If $X\subset M$ is a uniformity block, we denote by

$$     Q(n,\rho,\beta, X) = \inf \left\{    \sum_{x\in E}  \exp(S_n\phi(x))   |         E\subset X \text{ is }  \mu-(n,\rho,\beta)  \text{-spanning}          \right\}  .   $$ 
Restricting to a uniformity block is an essential modification of the quantity studied in \cite{Mendoza}. Because $\phi$ is not continuous, Theorem 1.1 in  \cite{Mendoza} does not apply.
Instead, we can prove the following.

\begin{lem}\label{lem:partsum}
If $\eta >0$ and $X\subset Z$ is a uniformity block of tolerance $\eta$, then there exists $m_0\in\mathbb{N}$ such that:
\begin{enumerate}
\item for all $m\geq m_0$ and any $x\in X$ one has\footnote{ We use the notation $A = \pm B$ to mean $  -B \leq  A  \leq B  $.} 
$$    \exp( S_m\phi (x)) =   | \det Df^m(x)_{| E^u(x)}  | ^{-d} =     e^{  \pm 3 \eta m d   \dim E^u}     \prod_{i=1}^{l}   e^{-\chi_i    d n_i   m };$$
\item for all $m\geq m_0$ and any finite set $A\subset X$,
$$  \frac{1}{m} \log  \sum_{x\in A}   \exp(S_m \phi (x)) > -3\eta d \dim(E^u) + \left[  \frac{1}{m} \log   \# A   - h_{\mu}(f)  \right] ; $$
\item for all $m\geq m_0$ and for any $\beta\in (0,1)$ one has
$$  \frac{1}{m} \log   Q(m,\rho,\beta, X) > -3\eta d \dim(E^u) + \left[  \frac{1}{m} \log   C(m,\rho,\beta)   - h_{\mu}(f)  \right] , $$
\item and in particular
$$  \lim_{\rho\to 0}   \liminf_{m\to \infty} \frac{1}{m} \log   Q(m,\rho,\beta, X) > -3\eta d \dim(E^u).$$
\end{enumerate}
\end{lem}

\begin{proof}
By Theorem \ref{thm:Pes}, for any $x\in Z$ we have that
$$  Df(x)_{| E^u(x)} = C^{-1}(f(x))  diag[   A_1(x), \ldots  ,  A_l(x)] C(x),  $$
consequently, since $Z$ is $f$-invariant,

\begin{align*}
  Df^m(x)_{| E^u(x)} & =    Df(f^{m-1}(x))   _{| E^u(f^{m-1}(x))}  \circ   \ldots   \circ    Df(f(x))   _{| E^u(f(x))}   \circ    Df(x) _{| E^u(x)}   \\
&= C^{-1}(f^m(x))  diag[  \prod_{k=1}^{m-1}  A_1(f^k(x)), \ldots  , \prod_{k=1}^{m-1}  A_l(f^k(x))] C(x).
\end{align*}

We then have that

$$   | \det Df^m(x)_{| E^u(x)}  | =  |\det C^{-1}(f^m(x)) | \cdot   \prod_{k=1}^{m-1} |\det  A_1(f^k(x))|\cdot \ldots  \cdot \prod_{k=1}^{m-1}  | \det A_l(f^k(x)) | \cdot |\det C(x)|. $$

We estimate all the terms in the product above.
By Statement (\ref{eqn:temp}) in Theorem \ref{thm:Pes} and because $X$ is a uniformity block, there exists $\tilde{m}_0\in\mathbb{N}$ such that
 $||   C( f^m(x))||, ||C^{-1}( f^m(x))||=  e^{ \pm \eta m}$ for all $m\geq \tilde{m}_0$. Consequently, by Hadamard's determinant inequality,
$$    |\det C^{-1}(f^m(x)) | =  e^{ \pm \eta m   \dim E^u}  \text{ for all $m\geq \tilde{m}_0$}. $$
By continuity of $C$ and $C^{-1}$ on $X$, there exists a constant $Q>0$ such that $  |\det C(x)|   = Q^{\pm 1}. $
In addition, for $i=1,\ldots, l$ and any $k\in\mathbb{N},$ Statement (\ref{eqn:norms}) of Theorem \ref{thm:Pes} implies that 
$$  |\det A_i(f^k(x))|   = e^{(\chi_i \pm \eta) n_i}. $$
Together we obtain that

$$     | \det Df^m(x)_{| E^u(x)}  | =  e^{ \pm \eta m   \dim E^u}   \prod_{i=1}^{l}   e^{(\chi_i \pm \eta) n_i   m } Q^{\pm 1}   
 =    e^{ \pm 2 \eta m   \dim E^u} Q^{\pm 1}    \prod_{i=1}^{l}   e^{\chi_i  n_i   m }  $$
for every $x\in X$ and $m\geq \tilde{m}_0$. We set $m_0\geq \tilde{m}_0$ so that $Q< e^{\eta \dim E^u m_0}.$
Then for any $m\geq m_0$ one has
$$     \exp (S_m\phi (x)) =   | \det Df^m(x)_{| E^u(x)}  | ^{-d} =     e^{  \pm 3 \eta m d   \dim E^u}     \prod_{i=1}^{l}   e^{-\chi_i    d n_i   m } $$
so we proved the first statement of the lemma.
If now $    A \subset X$ is any finite set, then for all $m\geq m_0$ we can estimate

$$  \sum_{x\in A} \exp (S_m\phi (x)) = \sum_{x\in A }    | \det Df^m(x)_{| E^u(x)}  | ^{-d}   \geq     e^{ -  3 \eta m d   \dim E^u}    \prod_{i=1}^{l}   e^{-\chi_i    d n_i   m }  \# A.   $$

From this we obtain

$$   \frac{1}{m} \log   \sum_{x\in A} \exp( S_m\phi (x))    \geq    - 3 \eta  d   \dim E^u      - \sum_{i=1}^l  \chi_i    d n_i    +   \frac{1}{m} \log \# A$$
$$ =    - 3 \eta  d   \dim E^u    + \left[  - \sum_{i=1}^l  \chi_i    d n_i  + h_{\mu}(f)\right]   +  \left[ \frac{1}{m} \log  \# A  - h_{\mu}(f)\right]. $$

We complete the proof of the second statement of the lemma by observing that $   - \sum_{i=1}^l  \chi_i    d n_i  + h_{\mu}(f) = 0$ by Theorem \ref{thm:LY}.

If now $    E\subset X \text{ is }  \mu-(m,\rho,\beta)  \text{-spanning},$ then $\# E\geq   C(m,\rho,\beta) $ so Statement (3) is a consequence of Statement (2).
In addition, by Theorem \ref{thm:entr}, we conclude

\begin{align*}
 \lim_{\rho \to 0} \liminf_{m\to \infty}   \frac{1}{m} \log   Q(m,\rho,\beta, X)    &   \geq   \lim_{\rho\to 0}  \liminf_{m\to\infty}  \left[   - 3 \eta  d   \dim E^u- \sum_{i=1}^l  \chi_i    d n_i    +   \frac{1}{m} \log  C(m,\rho,\beta) \right] \\
 & = - 3 \eta  d   \dim E^u     - \sum_{i=1}^l  \chi_i    d n_i  + h_{\mu}(f) =  - 3 \eta  d   \dim E^u.
\end{align*}

\end{proof}






The following is the analogue of Lemma 8.6 in \cite{ACW}.

\begin{lem}\label{lem:Y}
Let $X\subset Z$ be a uniformity block of tolerance $\eta>0$.
For any $\delta_0>0$ there exists $\bar{\rho}>0$ such that the following holds.
For any $\epsilon>0$ and $0< \rho< \bar{\rho}$ there exists $L\in\mathbb{N}$  (arbitrarily large), an open set $B\subset M$ of diameter less than $\epsilon$, and a set $\tilde{X}\subset B  \cap X$ such that:
\begin{enumerate}
\item $\tilde{X}$ is $(L,\rho)$-separated,

\item $f^L(\tilde{X})\subset B\cap X$,

\item $$\sum_{x\in \tilde{X}} \exp(S_L\phi(x)) \geq \exp[(-3\eta d \dim(E^u) - \delta_0)L]. $$

\end{enumerate}
\end{lem}

\begin{remark}
The relation between the parameters $\delta, \delta_1$ in Theorem B, $\delta_0$ in Lemma \ref{lem:Y}, and $\eta$ in Theorem \ref{thm:Pes} is stated at the end of the proof of Theorem B.
\end{remark}

\begin{proof}
Let $\epsilon, \delta_0>0 $ be fixed. We then consider $\xi :=  \frac{\delta_0}{  h_{\mu}(f)   + 4 }$.
Let $\alpha = \alpha_1, \ldots, \alpha_t$ be a finite cover (of cardinality $t$) of $X$ by open sets of diameter less than $\epsilon$. 
For $m\in \mathbb{N}$ define
$$  X_m := \{  x\in X | x \text{ and } f^n(x) \text{ are in the same element of }\alpha \text{ for some }  n\in [m, (1+\xi) m]    \}.   $$
The Birkhoff Ergodic Theorem implies that $\mu(X_m) \to \mu(X)$ as $m\to \infty$. Let $M_0$ be such that $ \mu(X_m)> \mu(X)/2  $ for all $m\geq M_0$.
Let $\bar{\rho}>0$ be small enough so that $$  \liminf_{m\to \infty}   \frac{1}{m} \log   C(m,\rho,1-\mu(X)/2)   \geq    h_{\mu}(f)    - \xi/2 \text{ for every } 0< \rho\leq 2 \bar{\rho}.$$
Let $M_1 > m_0$ (where $m_0$ is the constant in Lemma \ref{lem:partsum}) be big enough so that $$\frac{1}{m} \log   C(m,2\bar{\rho},1-\mu(X)/2)   \geq  h_{\mu}(f) - \xi \text{ for all } m\geq M_1.$$

We now fix $m> \max(  M_0, M_1, \xi^{-1}\log t    ).$ Chose any $\rho \in (0 ,\bar{\rho})$.
Let $E_m$ be a maximal $(m,\rho)$-separated set in $X_m$.
Then $E_m$  is $\mu-(m,2 \rho,1-\mu(X)/2)$-spanning, consequently
$$\# E_m \geq  \exp ((h_{\mu}(f) - \xi)m ).   $$



For $n\in [m, (1+\xi)m]$ consider the set
$$   V_n :=  \{  x\in E_m | x \text{ and } f^n(x) \text{ are in the same element of }\alpha   \}.      $$
Let $L$ be a value of $n$ maximizing $\# V_n$. 
From the definition of $X_m$ it follows that $E_m=\bigcup_{n=m}^{\lfloor  (1+\xi)m  \rfloor} V_n,$ where $\lfloor \cdot \rfloor$ denotes the floor function.
Consequently, $$  \xi m \cdot \# V_L  \geq \sum_{n=m}^{\lfloor  (1+\xi)m  \rfloor}  \#V_n \geq \# E_m.$$
Then

$$\# V_L\geq    \frac{ \# E_m }{  \xi m    }     \geq     \frac{    \exp ((h_{\mu}(f) - \xi)m )   }{\xi m}   \geq    \exp ((h_{\mu}(f) - 2\xi)m ).    $$ 
Consequently,

\begin{align*}
\frac{1}{L} \log \#V_L   - h_{\mu}(f) &=  (h_{\mu}(f) - 2\xi)\frac{m}{L}   - h_{\mu}(f) \geq   (h_{\mu}(f) - 2\xi)\frac{1}{1+\xi}   - h_{\mu}(f) \\
& =   \frac{    h_{\mu}(f) - 2\xi -   h_{\mu}(f) - \xi  h_{\mu}(f)    }{1+\xi}  =\frac{    - 2\xi - \xi  h_{\mu}(f)    }{1+\xi} \geq    - 2\xi - \xi  h_{\mu}(f) . 
\end{align*}

By Statement (2) in Lemma \ref{lem:partsum}, we obtain that

$$\sum_{x\in V_L} \exp(S_L\phi(x))   \geq     \exp   [  (-3\eta d \dim(E^u) -  2 \xi  - \xi h_{\mu}(f)   )L ] . $$ 

We now choose an element $\bar{\alpha}$ of $\alpha$ which maximizes the sum $\sum_{x\in V_L \cap \bar{\alpha}} \exp(S_L\phi(x)) .$
We set $\tilde{X}:= V_L \cap \bar{\alpha}$.
Observe that $\tilde{X} \subset V_L\subset E_m$ is $(m,\rho)$-separated, then it is also $(L,\rho)$-separated since $L\geq m$.
In addition, we have that
\begin{align*}
 \sum_{x\in    \tilde{X}} \exp(S_L\phi(x)) & \geq   \frac{  \exp [  (-3\eta d \dim(E^u) -  2 \xi  - \xi h_{\mu}(f)   )L ] }{t}\\
&  \geq  \exp [  (-3\eta d \dim(E^u) -  3 \xi  - \xi h_{\mu}(f)   )L ],  
\end{align*}
where we used that fact that $L\geq m >   \xi^{-1}\log t  .$
Finally, since $\xi = \frac{\delta_0}{h_{\mu}(f) +4},$ we conclude that

 $$\sum_{x\in    \tilde{X}} \exp(S_L\phi(x)) \geq     \exp [  (-3\eta d \dim(E^u)    -   \delta_0   )L ].$$
\end{proof}

If now $\tilde{X}$ is a set constructed in Lemma \ref{lem:Y}, one
can define the shift space $\mathcal{X}=\tilde{X}^{\mathbb{Z}}$ over the alphabet $\tilde{X}$. 
To each sequence  $\textbf{x}=(x_n)_{n\in\mathbb{Z}}\in \mathcal{X}$ one can associate a pseudo-orbit

$$   o(\textbf{x})  := \ldots  ,   x_0, \ldots, f^{L-1}(x_0), x_1 , \ldots , f^{L-1}(x_1), \ldots.   $$

Observe that for each $n\in\mathbb{Z}$ we have that $x_n\in \tilde{X}$ and $d(f^L(x_n), x_{n+1})<\epsilon$, where $\epsilon>0$ was chosen in Lemma \ref{lem:Y} as an upper bound for the diameter of the set $B$ and can be made arbitrarily small.  
The following was proved in \cite[Theorem 8.3]{ACW}.

\begin{thm}\label{thm:shad}
There exists $\kappa>0$ such that if $\eta>0$ is sufficiently small, then there are constants  $C_0,\epsilon_0>0$ such that if $\epsilon\in (0,\epsilon_0)$, then:
\begin{enumerate}
\item for each sequence   $\textbf{x}\in \mathcal{X}$ there is a unique point $y\in M$ whose orbit $C_0\epsilon$-shadows $o(\textbf{x})$,
\item $y$ belongs to the regular neighborhood $\psi^{-1}_{x_0}(B(0, r(x_0)/2)),$
\item If $y,y'$ shadow two pseudo-orbits $o(\textbf{x}), o(\textbf{x}')$ such that $x_n=x'_n$ for $|n|\leq N,$ then $d(y,y')\leq C_0 e ^{-\kappa N}$.
\end{enumerate}
\end{thm}
We consider $\Lambda_0$ - the set of all the orbits that shadow the elements of $\mathcal{X}$.
We claim that the set  $$\Lambda= \Lambda_0 \cup f(\Lambda_0)\cup \ldots \cup f^{L-1}(\Lambda_0)$$ satisfies the assertion of Theorem B.



Let $y\in \Lambda_0 $ be the unique point that shadows $o(\textbf{x})$.  For  $i=1,\ldots, l$, where $l$ is the number of distinct Lyapunov exponents for $f$,
we define cones  $\mathcal{C}^i(f^m(y))$ along the trajectory of $y$ in the following way.
 For $k\in [0,L-1]$, we define a cone $\mathcal{C}^i(  f^{nL + k  }   (y))$ in $ T_{f^{nL + k  }   (y)} M$
as the parallel transport of the cone $\mathcal{C}^i   ( E^i (f^k x_n), \pi/4   )$ - the cone at $ T_{f^k (x_n)}M$ of angle $\pi/4$ around the subspace  $E^i (f^k x_n)$.

Let $\delta_1$ be a small positive number. It was shown in  \cite{ACW} that by choosing $\eta>0$ sufficiently small, one can guarantee that such defined cones are preserved under $Df$ along the orbit of $y$.
In addition, there exist  subspaces $E^i     (f^ m( y))   \subset \mathcal{C}^i(  f^{m }   (y))  $ such that:
\begin{itemize}
\item $ T_{\Lambda}M = E^s \oplus E^1 \oplus E^2 \oplus \ldots \oplus E^l$ and for $i=1,\ldots, l,~ y\in\Lambda:$
\item $\dim (E^i) = n_i,$
\item $Df_{f^m(y)} E^i     (f^ m( y))   = E^i     (f^{ m+1}( y)),$
\item there exists $m_1$ such that for all $m\geq m_1$ and for all $n\in\mathbb{Z}$, and a vector $v\in E^i( f^{nL}(y))$ one has that
$$  \exp ((\chi_i - \delta_1/2)m)  \leq || Df^{m}_{f^{nL}y}(v)  || \leq  \exp ((\chi_i + \delta_1/2)m)  $$
and then choosing large enough $\tilde{m}_1> m_1$, for all $m\geq \tilde{m}_1, $  all $n\in\mathbb{Z}$, $k\in[0,L-1],$ and a vector $v\in E^i( f^{nL+k}(y))$ one has that
$$  \exp ((\chi_i - \delta_1)m)  \leq || Df^{ m}_{f^{nL+k}y}(v)  || \leq  \exp ((\chi_i + \delta_1)m).  $$
\end{itemize}
Then the function $\phi = -d \log |\det Df_{| E^u}| $ is defined along the trajectory of $y$. We conclude that $\phi$ is defined on the compact invariant set
 $\Lambda= \Lambda_0 \cup f(\Lambda_0)\cup \ldots \cup f^{L-1}(\Lambda_0)$.
 Theorem 5.3.2 in \cite{pesRed} gives that the subspaces $E^u(y)$ vary H\"older continuously with the point $y\in\Lambda$ and so $\phi$ is H\"older continuous on $\Lambda$.

In addition, for $y\in\Lambda$ and $m\geq \tilde{m}_1$  we conclude that $Df^m(y)_{| E^u(y)}$ has the block diagonal form:

$$Df^m(y)_{| E^u(y)}=  diag[ B^{(m)}_1(y), \ldots, B^{(m)}_l(y)  ] ,$$

with $B^{(m)}_i(y)\in GL(n_i, \mathbb{R})$, $i=1, \ldots ,l,$ such that

$$  e^{(\chi_i-\delta_1)m} < ||(B^{(m)}_i)^{-1}||^{-1} \leq ||B^{(m)}_i|| < e^{(\chi_i+\delta_1)m} $$

Using  Hadamard's determinant inequality we can estimate for all $y\in \Lambda$ and $m\geq \tilde{m}_1$:

\begin{equation}\label{eqn:had}
   \exp( S_{m} \phi (y)) =     |\det Df^m(y)_{| E^u(y)}|^{-d}   =e^{\pm \delta_1 d \dim E^u m } ~ \prod_{i=1}^l e^{-\chi_i d n_i m } .  
\end{equation}

If $y\in \Lambda_0$ then by Statement (1) in Lemma \ref{lem:partsum}, we can write 

$$  \exp (S_{mL} \phi (y)) = e^{\pm (\delta_1 + 3 \eta) d \dim E^u Lm}  \prod_{k=0}^{m-1}  \exp( S_L \phi (x_k) ) .    $$

Now consider the collection $\mathcal{I}_m$ of all distinct $m$-tuples $(x_0, \ldots, x_{m-1})\in \tilde{X}^{m}.$
To each element in $\mathcal{I}_m$ we can assign a point $y\in \Lambda_0$ that shadows it. Denote the corresponding set of $y$'s by $Y_m$.
 Since the set $\tilde{X}$ is $(L, \rho)$-separated, choosing $\epsilon>0$ small enough so that $C_0\epsilon< \rho/3$ we can guarantee that if
$y,y'$ are two points shadowing two distinct  $m$-tuples: $(x_0, \ldots, x_{m-1})$ and $(x'_0, \ldots, x'_{m-1})$, then $y$ and $y'$ are $(mL, \rho/3)$-separated. That means that the set $Y_m$ is $(mL, \rho/3)$-separated.
In addition, we have that
\begin{align*}
\sum_{y\in Y_m}   \exp( S_{mL} \phi (y)) &=   e^{\pm (\delta_1 + 3 \eta) d \dim E^u Lm}  \sum_{ (x_0, \ldots, x_{m-1}) \in \mathcal{I}_m   }  \prod_{k=0}^{m-1}  \exp( S_L \phi (x_k)) \\
& =       e^{\pm (\delta_1 + 3 \eta) d \dim E^u Lm}      \left(  \sum_{x_k \in \tilde{X}}    \exp( S_L \phi (x_k)) \right)^m    .  
\end{align*}

By Statement (3) in Lemma \ref{lem:Y}, we conclude

$$     \sum_{y\in Y_m}   \exp (S_{mL} \phi (y) )   \geq   \exp \left[        ( (-\delta_1 - 6\eta) d \dim E^u - \delta_0) Lm      \right]$$
 
and
\begin{align*}
 P(f_{|\Lambda},   - d   \log |  \det  Df_{| E^u}| )  &=  P(f_{|\Lambda_0},   - d   \log |  \det  Df_{| E^u}| ) \\
&>    (-\delta_1 - 6\eta) d \dim E^u - \delta_0.  
\end{align*}


Observe that $\Lambda$ is a compact, invariant hyperbolic set for $f$ and therefore there exists a unique equilibrium measure $\nu$ for $\phi$ with respect to $f_{| \Lambda}$.
This measure is supported on $\Lambda$ and in particular it is hyperbolic. Because $\nu$ is an equilibrium measure, choosing sufficiently large $m> 0$, we can write that that

\begin{equation}\label{eqn:rough}
    (-\delta_1 - 6\eta) d^u_{\mu} - \delta_0 < P_{\nu}(f,\phi) =  h_{\nu}(f)  -   \int d  \log |  \det  Df_{| E^u}| d\nu 
\end{equation}  
$$= h_{\nu}(f)  -  \frac{1}{m}\int d  \log |    Df^m_{| E^u}| d\nu  =  h_{\nu}(f) -  \left( \sum_{i=1}^l \chi_i d n_i  \right) \pm \delta_1 d^u_{\mu}.$$

On the other hand, since $\nu$ is hyperbolic, by Theorem \ref{thm:LY}, for each $i=1,\ldots, l$ there are:

\begin{itemize}
\item numbers $ 1 \leq k_i\leq n_i$ such that for $  j=1, \ldots, k_i  $ there are
\item $ 1\leq n_i^j\leq n_i$ with $\sum_{j=1}^{k_i}n_i^j = n_i$,
\item $\chi_i^j = \chi_i \pm \delta_1$,
\item $0 < d_i^{j} \leq n^j_i $ with $\sum_{i=1}^{l}\sum_{j=1}^{k_i} d_i^j = d^u_{\nu},$
 and such that
\end{itemize}

\begin{equation}\label{eqn:precise}
h_{\nu}(f) - \sum_{i=1}^l \sum_{j=1}^{k_i}   d_i^j \chi^j_i = 0.
\end{equation}

Comparing the equations (\ref{eqn:rough}) and (\ref{eqn:precise}) we obtain that

$$        \sum_{i=1}^l \sum_{j=1}^{k_i}   d_i^j \chi^j_i  =      h_{\nu}(f)   =   \left( \sum_{i=1}^l \chi_i d n_i  \right) \pm (2 \delta_1 + 6 \eta)d^u_{\mu} \pm \delta_0.      $$

Substituting  $\chi_i^j = \chi_i \pm \delta_1$, we can write

$$  \left(     \sum_{i=1}^l      \sum_{j=1}^{k_i} d_i^j        \chi_i  \right)  \pm d^u_{\nu} \delta_1    =   \left( \sum_{i=1}^l \chi_i d n_i  \right)  \pm (2 \delta_1 + 6 \eta)d^u_{\mu} \pm \delta_0.     $$

Combining like terms we obtain that

$$      \sum_{i=1}^l   \chi_i \left(    \left( \sum_{j=1}^{k_i} d_i^j     \right) -   d n_i \right)   =   \pm (3 \delta_1 + 6 \eta)d^u_{\mu} \pm \delta_0.$$

Denote $s:= (3 \delta_1 + 6 \eta)d^u_{\mu} + \delta_0$ and $\gamma_i := \sum_{j=1}^{k_i} d_i^j  . $ We estimate

\begin{align*}
 -s \leq \sum_{i=1}^l \chi_i (\gamma_i - d n_i) &= \sum_{\{   i | \gamma_i > d n_i \}}    \chi_i (\gamma_i - d n_i)     +       \sum_{\{   i | \gamma_i \leq d n_i \}}     \chi_i (\gamma_i - d n_i)    \\
 &\leq   \sum_{\{   i | \gamma_i > d n_i \}}    \chi_l (\gamma_i - d n_i)     +       \sum_{\{   i | \gamma_i \leq d n_i \}}     \chi_1 (\gamma_i - d n_i)   .
\end{align*}

It follows that 

$$      \sum_{\{   i | \gamma_i \leq d n_i \}}     (\gamma_i - d n_i)  \geq -\frac{s}{\chi_1} - \frac{\chi_l}{\chi_1}   \sum_{\{   i | \gamma_i > d n_i \}}    (\gamma_i - d n_i) . $$

Consequently,
\begin{align*}
 d^u_{\nu} - d^u_{\mu} &=    \sum_{i=1}^l  (\gamma_i - d n_i) =  \sum_{\{   i | \gamma_i > d n_i \}}     (\gamma_i - d n_i)     +       \sum_{\{   i | \gamma_i \leq d n_i \}}     (\gamma_i - d n_i)  \\
&   \geq           \sum_{\{   i | \gamma_i > d n_i \}}     (\gamma_i - d n_i)       -\frac{s}{\chi_1} - \frac{\chi_l}{\chi_1}   \sum_{\{   i | \gamma_i > d n_i \}}    (\gamma_i - d n_i) .          \\
&  \geq        -\frac{s}{\chi_1} - \left(  \frac{\chi_l}{\chi_1}  - 1  \right)(1-d) \dim E^u.   
\end{align*}

Where we estimated 

$$    \sum_{\{   i | \gamma_i > d n_i \}}    (\gamma_i - d n_i)  \leq   \dim E^u - d \dim E^u.   $$

Finally, this gives that

$$   d^u_{\nu} \geq    -\frac{s}{\chi_1} + \dim E^u \left(   1+  (d-1) \frac{\chi_l}{\chi_1}   \right).$$

To finish the proof of Theorem B it is enough to choose $\delta_0, \delta_1$ and $\eta$ small enough so that

$$   \frac{  (3 \delta_1 + 6\eta)d \dim E^u + \delta_0     }{\chi_1}   < \delta.  $$


\section{Emergence}\label{sec:emer}

In this section we exploit the symbolic representation of the set $\Lambda_0$ constructed in the previous section. 
Recall that $\tilde{X}$ is the set constructed in Lemma \ref{lem:Y} and we consider the shift space $\mathcal{X}=\tilde{X}^{\mathbb{Z}}$ over the alphabet $\tilde{X}$. 
For $x,x'\in \tilde{X}$ set $\delta(x,x')=1$ if $x\neq x'$ and $\delta(x,x)=0$. We can then define the metric $d_{\mathcal{X}}$ on $\mathcal{X}$ as follows,
$$d_{\mathcal{X}}(\textbf{x},\textbf{x}'):= \sum_{n\in \mathbb{Z}}  \frac{\delta(x_n,x'_n)}{2^n}.$$
As before, to each  $\textbf{x}\in \mathcal{X}$ we associate a pseudo-orbit
$$   o(\textbf{x})  := \ldots  ,   x_0, \ldots, f^{L-1}(x_0), x_1 , \ldots , f^{L-1}(x_1), \ldots.   $$
Finally, we consider $\Lambda_0$ - the set of all the orbits that shadow the elements of $\mathcal{X}$.
We have the following.
\begin{lem}\label{lem:hor}
If $\eta, \epsilon>0$ are sufficiently small, then:
\begin{enumerate}
\item$f^L_{|\Lambda_0}$ is topologically conjugate to the full shift in $\# \tilde{X}$ symbols,
\item the conjugacy $\pi : \mathcal{X} \to \Lambda_0$ is H\"older continuous with respect to the metric $d_{\mathcal{X}}$ on $\mathcal{X}$.
\end{enumerate}
\end{lem} 

\begin{proof}
Let $\epsilon>0$ be small enough so that $C_0 \epsilon < \rho/3$.  Since the set $\tilde{X}$ is $(L, \rho)$-separated,
if $y,y'$ are two points shadowing two $\epsilon$-pseudo orbits $o(\textbf{x})$, $o(\textbf{x}')$ with $x_m\neq x'_m$, then $y$ and $y'$ are $(mL, \rho/3)$-separated and in particular different.
That means that the map $\pi : \mathcal{X} \to \Lambda_0$ defined by assigning to each sequence $\textbf{x}$ a unique point $y\in \Lambda_0$ that shadows $o(\textbf{x})$ is one-to-one.
The fact that $\pi$ is H\"older continuous follows from Statement (3) in Theorem \ref{thm:shad}.
\end{proof}

Choose $\bar{y}\in\Lambda_0$ and consider the set $ Y:=  V^u(\bar{y})\cap \Lambda_0$, where $V^u(\bar{y})$ denotes the local unstable manifold at $\bar{y}$.
Then given  $z\in \Lambda_0$ we denote by $V^s(z)$ the local stable manifold at $z$ and define $\theta(z)= V^s(z)\cap V^u(\bar{y})  \in Y$. 
We consider the map $F: Y \to Y$ defined by $F(y)= \theta( f^L(y))$.
Observe that: 

\begin{itemize}
\item  if a point $y\in Y$ has high pointwise emergence with respect to $F$, then it also has high pointwise emergence with respect to $f$,

\item if  $y\in Y$ has high pointwise emergence with respect to $f$, then so does every point $z$ on its stable manifold, 

\item by Lemma \ref{lem:hor}, $F$ is topologically conjugate to the left shift $(\sigma^+,\mathcal{X}^+)$, where we denoted  $\mathcal{X^+} = \tilde{X}^{\mathbb{N}}$,

\item the conjugacy $\pi^+ : \mathcal{X^+}  \to Y$  is H\"older continuous.
\end{itemize}

In addition to the above observations, we prove the following.

\begin{lem}\label{lem:1}
If $\omega^+\in \mathcal{X}^+$ has high pointwise emergence with respect to $\sigma^+$, then $\pi^+(\omega^+)$ has high pointwise emergence with respect to $F$.
\end{lem}

We conclude that 

\begin{equation}\label{eqn:low}
\dim_H(\mathcal{E}_f) \geq  \dim E^s  +    \dim_H  (\pi^+(E^+)),
\end{equation}

where $E^+$ denotes the set of points with high emergence for $\sigma^+$. We estimate $\dim_H  (\pi^+(E^+))$ and conclude Theorem A in the next section.

\begin{proof}[Proof of Lemma \ref{lem:1}]
We consider the metric $d^+(\textbf{x},\textbf{x}'):=\sum_{j=0}^{\infty} \frac{\delta(x_n,x'_n)}{2^n}$ on $\mathcal{X}^+$.
If $m$ is the smallest natural number for which $x_m\neq x'_m,$ then $d^+(\textbf{x},\textbf{x}')\leq \frac{C}{2^m}$ for some uniform constant $C>0$.
At the same time, for $y= \pi^+(\textbf{x}), y'=\pi^+(\textbf{x}')$ we must have that $d(f^k(y),f^k(y'))>\rho/3$ for some $(m-1)L \leq k\leq mL$.
Consequently, 
$$   d(y,y' )  >  \frac{\rho}{3} ~ \max_{x\in\Lambda}||  Df_{| E^u(x)}||^{-mL}  \geq \frac{\rho}{3C}~   d^+(\textbf{x},\textbf{x}')^K, $$
 where  $K= L \log_2 \max_{x\in\Lambda}||Df_{| E^u(x)}||.$    

Let now $\nu$ and $\mu$ be two Borel probability measures on $Y$ and let $\mathbb{P}$ be a Borel probability measure on $Y\times Y$ with the property that
$\mathbb{P}\circ \pi_1^{-1} = \nu$ and $\mathbb{P}\circ \pi_2^{-1} = \mu$. Here $\pi_i$ is the canonical projection on the $i$'th coordinate.
We consider the map $\pi^+_*: \mathcal{X}^+\times \mathcal{X}^+ \to Y \times Y$ given by $\pi^+_*(\textbf{x},\textbf{x}')= (\pi^+(\textbf{x}),\pi^+(\textbf{x}')).$
We denote $\mathbb{P}_*= \mathbb{P}\circ \pi^+_*$.
Observe that $\mathbb{P}_*\circ \pi_1^{-1} = \mu\circ \pi^+$ and  $\mathbb{P}_*\circ \pi_2^{-1} = \nu\circ \pi^+$.
Consequently,
\begin{align*}
\int_{Y\times Y} d(y,y') d\mathbb{P}& = \int_{\mathcal{X}^+\times \mathcal{X}^+}     d(\pi^+(\textbf{x}),\pi^+(\textbf{x}'))   d\mathbb{P}_*  \geq     \int_{\mathcal{X}^+\times \mathcal{X}^+}      \frac{\rho}{3C}~   d^+(\textbf{x},\textbf{x}')^K   d\mathbb{P}_*   \\
&\geq     \inf_{\mathbb{P}_* }   \int_{\mathcal{X}^+\times \mathcal{X}^+}     \frac{\rho}{3C}~   d^+(\textbf{x},\textbf{x}')^K   d\mathbb{P}_*   ,   
\end{align*}
where the infimum is taken over all Borel probability measures with the property that $\mathbb{P}_*\circ \pi_1^{-1} = \mu\circ \pi^+$ and  $\mathbb{P}_*\circ \pi_2^{-1} = \nu\circ \pi^+$.
Using Jensen's inequality, we obtain that
$$   \int_{Y\times Y} d(y,y') d\mathbb{P}    \geq     \frac{\rho}{3C}~ \left( \inf_{\mathbb{P}_* }   \int_{\mathcal{X}^+\times \mathcal{X}^+}      d^+(\textbf{x},\textbf{x}')   d\mathbb{P}_*   \right)^K  $$
and finally,
$$    d_W(\nu,\mu) \geq  \frac{\rho}{3C} d_W( \nu\circ \pi^+, \mu\circ \pi^+)^K.$$

Then for any $\epsilon>0$ and $y\in Y$ we have that

$$   \mathcal{E}_{F,y}\left(   \frac{\rho}{3C} \epsilon^K   \right)\geq  \mathcal{E}_{\sigma^+,(\pi^+)^{-1}(y)}(\epsilon).  $$

Consequently, if $\omega^+\in \mathcal{X}^+$ has high pointwise emergence with respect to $\sigma^+$, then for $y=\pi^+(\omega^+)$ we have that

$$   \limsup_{\epsilon \to 0}   \frac{\log     \mathcal{E}_{F,y}(   \frac{\rho}{3C} \epsilon^K   )    }{  - \log   \frac{\rho}{3C} \epsilon^K     }   \geq  \limsup_{\epsilon \to 0}   \frac{  \log  \mathcal{E}_{\sigma^+,\omega^+}(\epsilon)    }{   - \log   \frac{\rho}{3C} \epsilon^K   }  
=   \limsup_{\epsilon \to 0}   \frac{ \log   \mathcal{E}_{\sigma^+,\omega^+}(\epsilon)    }{   - \log   \frac{\rho}{3C} - K \log  \epsilon   }   = \infty.   $$

\end{proof}

\section{Hausdorff dimension}\label{sec:dim}

In this section we show that $  \dim_H  (\pi^+(E^+))   \geq   (d- d')\dim E^u$,  
for any  $$d' >   \frac{\delta + 6 \delta_1 d \dim E^u + (1-d)\sum_{i=1}^l(\chi_i -\chi_1) n_i}{\chi_1 \dim E^u}.  $$

Before we start on the proof of this claim, let us see how it implies Theorem A.

\begin{proof}[Proof of Theorem A]
We observe that by Theorem B, $\tilde{\tau}= \delta + 6 \delta_1 d \dim E^u$ can be made arbitrarily small 
and that ultimately $d'$ can be chosen arbitrarily close to $  \frac{ (1-d)\sum_{i=1}^l(\chi_i -\chi_1) n_i}{\chi_1 \dim E^u}.$
We then conclude that
$$  \dim_H(\mathcal{E}_f) \geq \dim(E^s) + \left(d-   \frac{ (1-d)\sum_{i=1}^l(\chi_i -\chi_1) n_i}{\chi_1 \dim E^u}\right)  \dim E^u    $$
$$ \geq     \dim(E^s) + \left(d-   \frac{ (1-d)(\chi_l -\chi_1) }{\chi_1 }\right) \dim E^u   \geq  \dim(E^s) +  \left( 1   - (1-d)\frac{\chi_l}{\chi_1}   \right)\dim E^u .   $$

\end{proof}

We prove the claim in the next two subsections. In \ref{sec:pressure} we establish some estimates related to the pressure of $\phi$. Then in \ref{sec:Haus-final} we use it to derive the final estimate on the Hausdorff dimension of $(\pi^+(E^+)$.

\subsection{Pressure estimates}\label{sec:pressure}
By Lemma A2.1  in \cite{pes97}, there exists a H\"older continuous function $\phi^+_{\mathbb{Z}}$ on $\mathcal{X}$, which is cohomologous to $S_L\phi\circ \pi$ and only depends on the positive side of the sequence.
By Proposition A2.2 in \cite{pes97}, denoting by $\phi^+$ the restriction of $\phi^+_{\mathbb{Z}}$ to $\mathcal{X}^+$, one has that 
$   P(\sigma, S_L\phi\circ \pi) = P(\sigma^+, \phi^+).  $
Let $E^+$ denote the set of points with high emergence for $\sigma^+$.
By Theorem 1.1 in \cite{NakZ}, $ P(\sigma^+_{| E^+}, \phi^+)= P(\sigma^+, \phi^+).$
We will use this to find a lower bound for $\dim_H (\pi^+  (E^+))$.

Let $d$ be as in Condition $\textbf{A1}$. Recall that $$    P(f_{|\Lambda_0},   - d   \log |    Df_{| E^u}| )  >   - \delta   ,$$
and let  $\tilde{\tau}= \delta + 6\delta_1 d \dim E^u. $
We choose $d'< d$ such that
\begin{equation}\label{def:d'}
d'>  \frac{\tilde{\tau} + (1-d)\sum_{i=1}^l(\chi_i -\chi_1) n_i}{\chi_1 \dim E^u}.
\end{equation}
We then have $d = d' + d''$ for some $d'' > 0.$
Since  $\phi^+_{\mathbb{Z}}$ is cohomologous to  $S_L\phi\circ \pi$, there exists a H\"older continuous function $u: Y \to \mathbb{R}$ such that
$$  \phi^+\circ (\pi^+)^{-1}(x) =    S_L\phi(x)  + u(F(x)) - u(x)  $$
for all $x\in Y$. Since $  P(f_{|\Lambda_0},   - d   \log |   \det  Df_{| E^u}| )>-\delta, $ then also \newline$  P(F_{|\pi^+(E^+)},    \phi^+\circ (\pi^+)^{-1}  )  > -\delta L$.
From the dimensional definition of topological pressure (see for example Section 11 in  \cite{pes97}) it follows that
\begin{align*}
\infty & = \lim_{N\to \infty}  \inf_{\mathcal{I}_F}  \sum_{(x,m)\in\mathcal{I}_F} \exp\left[  S_{F,m}     \phi^+\circ (\pi^+)^{-1} (x)    + \delta mL \right]\\
&=    \lim_{N\to \infty}  \inf_{\mathcal{I}_F}  \sum_{(x,m)\in\mathcal{I}_F} \exp\left[  S_{mL}     \phi(x) + u(f^{mL}(x))   - u(x)    + \delta mL \right]        ,  
\end{align*}
where the infimum is taken over all collections of pairs $(x,m)$ with $x\in Y, m\geq N$ such that $\pi^+(E^+) \subset \bigcup_{(x,m)\in\mathcal{I}_F} B_{F,m}(x,s)$ and $s>0$ is sufficiently small . 
The notation $S_{F,m}, B_{F,m}$ indicates that the summation and the Bowen balls are with respect to $F$. Recall that $S_m, B_m$ denote the summation and Bowen ball with respect to $f$.
By (\ref{eqn:had}), we continue

$$   \leq    \lim_{N\to \infty}  \inf_{\mathcal{I}}  \sum_{(x,m)\in\mathcal{I}} \exp\left[    \sum_{i=1}^{l}   (-\chi_i    d n_i   m)      +\delta_1 m d   \dim E^u     + U    + \delta m \right],        $$

where $U:= 2 \max_{x\in Y} |  u(x) |$ and the infimum is taken over all collections of pairs $(x,m)$ with $x\in Y, m\geq NL$ is a multiple of $L$, and such that $\pi^+(E^+) \subset \bigcup_{(x,m)\in\mathcal{I}} B_{m}(x,s)$.
For a fixed $N\in\mathbb{N}$ denote
$$   \zeta_N:=  \inf_{\mathcal{I}}  \sum_{(x,m)\in\mathcal{I}} \exp\left[    \sum_{i=1}^{l}   (-\chi_i    d n_i   m)      + \delta_1 m d   \dim E^u     + U    + \delta m \right].    $$

\subsection{Hausdorff dimension}\label{sec:Haus-final}
It is worth emphasizing that the definition of $\zeta_N$ above involves covers of the set  $\pi^+(E^+)$ by Bowen balls $B_m(x,s)$ with $m\to \infty$. In order to establish lower bound for the Hausdorff dimension of  $\pi^+(E^+)$ we need to consider covers by balls $B(x,s_m)$
- the balls with respect to the Riemannian metric on $M$ where the radius $s_m\to 0$. In what follows we will establish the relation between the two, which can be roughly explained as follows.
Given a Bowen ball  $B_m(x,s)$ we see that it has the smallest diameter in the direction of $E^l(x)$ and the biggest diameter in the direction of $E^1(x)$.
We can then enlarge it to the set $C_m(x,s)=B^1_m(x,s)\times B^2_{m_2}(x,s)\times \ldots \times B^l_{m_l}(x,s)$, where each $B^i_{m_i}(x,s)$ can be thought of as a \textit{'Bowen ball in the direction of'} $E^i$ (see (\ref{def:prod}) for precise definition),
and the numbers $m_i\in\mathbb{N}$ are chosen in such a way that the diameter of $C_m(x,s)$ in each direction is roughly the same, that is such sets can be used to approximate regular balls with respect to the Riemannian metric on $M$.
We will use covers by sets  $C_m(x,s)$ to estimate the Hausdorff dimension of $\pi^+(E^+)$.

In order to relate covers by sets $C_m(x,s)$ and covers by Bowen balls $B_m(x,s)$, we notice that we can cover each $C_m(x,s)$ with sets of the form $B^1_m(x',s)\times B^2_{m}(x',s)\times \ldots \times B^l_{m}(x',s)$
- those sets can be thought of as approximations of $B_m(x',s)$. We present a construction of such covers and in (\ref{eqn:refined}) we estimate their cardinality for each $C_m(x,s)$. Then in (\ref{eqn:zeta}) we use it to replace in $\zeta_N$ the covers by  Bowen balls $B_m(x,s)$  with covers by the sets $C_m(x,s)$. This eventually leads in (\ref{eqn:haus-final}) to our final estimate on the Hausdorff dimension of   $\pi^+(E^+)$.\\

We have the following foliated structure on $Y$ induced by "intermediate" unstable manifolds $V^1(x), \ldots, V^l(x)$ corresponding to distinct Lyapunov exponents:
$$ V^u(x)= V^1(x)  \supset  V^2(x)  \supset  \ldots  \supset   V^l(x),$$
with
$$  V^i(x):=  \{   y\in B(x, r(x)) ~ | ~    \limsup_{n\to \infty}  \frac{1}{n}  \log   d(f^{-n}(x), f^{-n}(y))   < - \chi_i  \}.   $$

By slightly abusing the notation, we will consider sets
$$  B^1_{m_1}(x,s) \times   \ldots \times B^l_{m_l}(x,s),    
\text{ where } B^i_{m_i}(x,s):= B_{m_i}(x,s)\cap V^i(x)$$
and $  B^1_{m_1}(x,s) \times   \ldots \times B^l_{m_l}(x,s)$ is defined as follows. 

We first denote

$$  T_l :=      B^l_{m_l}(x,s).  \text{ Then we define }  T_{l-1}:=   \bigcup_{x^{(l)}  \in T_l }        B^{l-1}_{m_{l-1}}(   x^{(l)},  s),    $$

$$ T_{l-2}:=   \bigcup_{x^{(l-1)}  \in T_{l-1} }        B^{l-2}_{m_{l-2}}(   x^{(l-1)},  s).    $$

Continuing in this manner we finally define $  B^1_{m_1}(x,s) \times   \ldots \times B^l_{m_l}(x,s)$ as

\begin{equation}\label{def:prod}
 B^1_{m_1}(x,s) \times   \ldots \times B^l_{m_l}(x,s)=  T_{1}:=   \bigcup_{x^{(2)}  \in T_{2} }        B^{1}_{m_{1}}(   x^{(2)},  s).       
\end{equation}

Observe that there exist $s_1, s_2>0$ such that denoting $$B^{\times}_m(x,s):=   B^1_{m}(x,s) \times   \ldots \times B^l_{m}(x,s) \text{  we have } $$

\begin{equation}\label{eqn:bx-relation}
             B_m(x, s_1)   \subset   B^{\times}_m(x,s)   \subset  B_m(x,s_2)      \text{  for every $m\in\mathbb{N}$ and $x\in Y$} .            
\end{equation}

Given $m\in\mathbb{N}$, and $x\in Y$, we set $m_i=m_i(m,x,s)$ for $i=2, \ldots, l$ such that 
\begin{equation}\label{eqn:diam}
  1   \leq  \frac{diam(      B^i_{m_i}(x,s)      )}{diam(    B^1_m(x,s)    )}  \leq     \max_{x\in \Lambda} || Df_{|E^u(x)}||.
\end{equation}

We observe that there exist uniform constants $c_1, c_2$ such that denoting $$C_m(x,s):=  B^1_{m}(x,s) \times  B^2_{m_2}(x,s) \times \ldots \times B^l_{m_l}(x,s)$$ one has that

\begin{equation}\label{eqn:C-relation}
 B(x, c_1   diam B^1_m(x,s))  \subset C_m(x,s)  \subset    B(x, c_2   diam B^1_m(x,s))   
\end{equation}
for every $ x\in Y$ and $m\in\mathbb{N}$.





Given $x\in Y$ and $m\in\mathbb{N}$ one can construct a cover of $C_m(x,s)$ by sets $B^{\times}_m(x',s)$ inductively as follows.

We first choose $\{   x^l_{j_l}   \} \subset  B^l_{m_l}(x,s)$ such that

$$  \bigcup_{j_l}     B^l_m(x^l_{j_l},s)      \supset  B^l_{m_l}(x,s).  $$

The elements  $\{   x^l_{j_l}   \}$ can be chosen in such a way that their number does not exceed
$  \left( \frac{diam B^l_{m_l}(x,s)}{1/3 \min_j diam    B^l_{m}(x^l_{j_l},s) }  \right)^{n_l} .$
By (\ref{eqn:diam}), we estimate

$$  \left( \frac{diam B^l_{m_l}(x,s)}{1/3 \min_j diam    B^l_{m}(x^l_{j_l},s) }  \right)^{n_l}  \leq  \left( \frac{diam B^1_{m}(x,s)   \max_{x\in\Lambda}||Df_{|E^u(x)}||}{1/3 \min_j diam    B^l_{m}(x^l_{j_l},s) }  \right)^{n_l} $$

$$  \leq      \left(\frac{\max_{x\in\Lambda}||Df_{|E^u(x)}||}{1/3 }\right)^{n_l}     \left( \frac{e^{(-\chi_1 +\delta_1)m}}{e^{(-\chi_l-\delta_1)m}}  \right)^{n_l}    \leq   e^{(\chi_l - \chi_1 + 3\delta_1)m n_l},$$

where in the last inequality we used the fact that $\frac{\max_{x\in\Lambda}||Df_{|E^u(x)}||}{1/3 } < e^{\delta_1 m}$ for large enough $m\in\mathbb{N}$.

Next, for each $x^l_{j_l}$ in the previous step, we find a collection

 $\{   x^{l-1}_{j_{l-1}}   \} \subset     B^{l-1}_{m_{l-1}}(x^l_{j_l},s)      $ such that

$$  \bigcup_{j_l}     \bigcup_{j_{l-1}}   B^{l-1}_m(x^{l-1}_{j_{l-1}},s)     \supset  T_{l-1}.  $$

The number of elements produced in this step (for each  $x^l_{j_l}$) can be estimated from above by

$$  \left( \frac{diam B^{l-1}_{m_{l-1}}(x,s)}{1/3 \min_j diam    B^{l-1}_{m}(x^{l-1}_{j_{l-1}},s) }  \right)^{n_{l-1}}    \leq   e^{(\chi_{l-1} - \chi_1 + 3\delta_1)m n_{l-1}}.$$






Continuing the construction in this manner we finally obtain that the set $ C_m(x,s)$ is covered by a union of no more than 

\begin{equation}\label{eqn:refined}
   \prod_i   e^{(\chi_i - \chi_1 + 3\delta_1)m n_i}  
\end{equation}

elements of the form  $B^{\times}_m(x^2,s)= B^1_{m}(x^2,s) \times  B^2_{m}(x^2,s) \times \ldots \times   B^l_m(x^2,s). $\\

We then continue with the estimate of $\zeta_N.$ Let $\mathcal{I}^{E}$ denote a cover of $\pi^+(E^+)$ by sets of the form $B^{\times}_m(x,s),$ where  $x\in Y$ and $ m\geq N$ is a multiple of $L$.
Let $\mathcal{I}^{S}$ denote a cover of $\pi^+(E^+)$ by sets of the form $C_m(x,s),$ where  $x\in Y$ and $ m\geq N$ is a multiple of $L$. By (\ref{eqn:bx-relation}), there exists a uniform constant $s'>0$ such that

$$ \zeta_N \leq  \inf_{\mathcal{I}^{E}} ~ s'   \sum_{B^{\times}_m(x,s)\in\mathcal{I}^E} \exp\left[    \sum_{i=1}^{l}   (-\chi_i    d n_i   m)      + \delta_1 m d   \dim E^u     + U    + \delta m \right] . $$
Observe that $s'e^U\leq e^{\delta_1 m}$ for large enough $m\in\mathbb{N}$. Then using (\ref{eqn:refined}) and denoting $\tau_1:=  \delta + 5\delta_1 d \dim(E^u) $ leads to 
\begin{equation}\label{eqn:zeta}
\zeta_N \leq     \inf_{\mathcal{I}^{S}}    \sum_{C_m(x,s)\in\mathcal{I}^S} \exp\left[    \sum_{i=1}^{l}   (-\chi_i    d n_i   m)      + \tau_1 m    +  \sum_{i=1}^l (\chi_i - \chi_1 )m n_i      \right] . 
\end{equation}
 
Writing $d= d' + d''$, where $d'$ was defined in (\ref{def:d'}),  we continue,
\begin{align*}
\zeta_N   &  \leq     \inf_{\mathcal{I}^{S}}    \sum_{C_m(x,s)\in\mathcal{I}^S}     \exp\left[      -\chi_1 d'' \dim E^u m       -\chi_1 d' \dim E^u m   + \tau_1 m + (1-d)m  \sum_{i=1}^l (\chi_i - \chi_1 ) n_i    \right]   \\
&\leq     \inf_{\mathcal{I}^{S}}    \sum_{C_m(x,s)\in\mathcal{I}^S}     diam B^1_m(x,s)^{d'' \dim E^u}    \exp\left[      -\chi_1 d' \dim E^u m   + \tilde{\tau} m + (1-d)m  \sum_{i=1}^l (\chi_i - \chi_1 ) n_i    \right] ,  
\end{align*}

In addition, by the choice of $d',$ we have that 

$$     \exp\left[      -\chi_1 d' \dim E^u m   + \tilde{\tau} m + (1-d)m  \sum_{i=1}^l (\chi_i - \chi_1 ) n_i    \right] < 1.  $$

Using this and (\ref{eqn:C-relation}), we obtain that for sufficiently large $N$,



\begin{equation}\label{eqn:haus-final}
\inf_{\{  B(x,\epsilon) \}    } \sum  diam(B(x,\epsilon))^{d'' \dim E^u} >     \zeta_N     ,    
\end{equation}

where the infimum is taken over all collections of open balls $B(x,\epsilon)$ with $\epsilon < \bar{\epsilon}_N$ covering $\pi^+(E^+)$, where $ \bar{\epsilon}_N= c_1\min_{x\in Y} diam( B_{N+L}^1(x,s)  )$.
By the definition of the Hausdorff dimension we conclude that $\dim_H(\pi^+(E^+) )  \geq d'' \dim E^u.  $

\end{document}